\documentclass{article}
\usepackage{amsmath,amsthm,amscd,amssymb}
\usepackage[mathscr]{eucal}
\usepackage{amssymb}
\usepackage{latexsym}

\theoremstyle{definition}
\newtheorem{Def}{Definition}[section]
\newtheorem{Thm}[Def]{Theorem}
\newtheorem{Prop}[Def]{Proposition}
\newtheorem{Rem}[Def]{Remark}

\newtheorem{Lem}[Def]{Lemma}

\numberwithin{equation}{section}

\begin{document}

\title{Note on Igusa's cusp form of weight 35}

\author{Toshiyuki Kikuta, Hirotaka Kodama and Shoyu Nagaoka}
\maketitle

%\date{}
 
%\dedicatory{}

%\keywords{}

\begin{abstract}
A congruence relation satisfied by Igusa's cusp form of weight 35 is presented.
As a tool to confirm the congruence relation, a Sturm-type theorem for the case of
 odd-weight Siegel modular forms of degree 2 is included.
\end{abstract}

\maketitle

%\noindent 
%{\bf Mathematics subject classification}: Primary 11F33 \and Secondary 11F46\\
%\noindent
%{\bf Key words}: 
\section{Introduction}
In \cite{Igusa1}, Igusa gave a set of generators of the graded ring
of degree 2 Siegel modular forms. In these generators, there are four even-weight forms
$\varphi_4$, $\varphi_6$, $\chi_{10}$, $\chi_{12}$, and only one odd-weight form $\chi_{35}$.
Here $\varphi_k$ is the normalized Eisenstein series of weight $k$, and $\chi_k$ is
a cusp form of weight $k$.

The purpose of this paper is to introduce a strange congruence relation of the odd-weight cusp
form $X_{35}$, which is a suitable normalization of $\chi_{35}$ (for the precise definition, 
see $\S$ \ref{generators}).\\
\\
\textbf{Main result.} Denote by $a(T;X_{35})$ the $T$-th Fourier coefficient of the
cusp form $X_{35}$. If $T$ satisfies  $\text{det}(T)\not\equiv 0 \pmod{23}$, then
\[
a(T;X_{35}) \equiv 0 \pmod{23},
\] 
or equivalently,
\[
\varTheta(X_{35}) \equiv 0 \pmod{23},
\]
where $\varTheta$ is the theta operator on Siegel modular forms
(for the precise definition, see $\S$ \ref{theta}). \\
\\
\quad This result shows that {\it almost} all the Fourier coefficients $a(T;X_{35})$
are divisible by 23. 
%%%%%%%%%%%%%%%%%%%%%%%%%%%%%%%%%%%%%%%%%%%%%%%%%%%%%%%%%%%%%%%%%%%%%%%%%%%%%%%%%%%%
\section{Preliminaries}
\subsection{Notation}
\label{notation}
%%%%%%%%%  
First we confirm the notation. Let $\Gamma _n=Sp_n(\mathbb{Z})$ be the Siegel modular group of degree $n$ and $\mathbb{H}_n$ 
the Siegel upper-half space of degree $n$. We denote by
$M_k(\Gamma _n)$ the $\mathbb{C}$-vector space of all
Siegel modular forms of weight $k$ for $\Gamma_n$, and $S_k(\Gamma _n)$ is the subspace of cusp forms.

Any $F(Z)$ in $M_k(\Gamma _n)$ has a Fourier expansion of the form
\[
F(Z)=\sum_{T\in L_n}a(T;F)q^T,\quad q^T:=e^{2\pi i\text{tr}(TZ)},
\quad Z\in\mathbb{H}_n,
\]
where $T$ runs over all elements of $L_n$, and 
\begin{align*}
\Lambda_n&:=\{ T=(t_{ij})\in Sym_n(\mathbb{Q})\;|\; t_{ii},\;2t_{ij}\in\mathbb{Z}\; \},\\
L_n&:= \{ T\in \Lambda _n \;|\; T\ \text{is\ semi-positive\ definite} \}.
\end{align*}
In this paper, we deal mainly with the case of $n=2$. For simplicity, we write 
\[
T=(m,n,r)\quad \text{for}\quad T=\begin{pmatrix} m & \frac{r}{2}\\
\frac{r}{2} & n \end{pmatrix}\in \Lambda _2.
\]

For a subring $R$ of $\mathbb{C}$, let $M_{k}(\Gamma _n)_{R}\subset M_{k}(\Gamma _n)$ denote the space of all modular forms 
whose Fourier coefficients lie in $R$. 
%%%%%%%%%%%%%%%%%%%%%%%%%%%%%%%%%%%%%%%%%%%%%%%%%%%%%%%%%%%%%%%%%%%%%%%%%%%%%%%%%%%%%%%%%%%%%%%%%%%
\subsection{Igusa's generators}
\label{generators}
Let
\[
M(\Gamma_2)=\bigoplus_{k\in\mathbb{Z}}M_k(\Gamma_2)
\]
be the graded ring of Siegel modular forms of degree 2. Igusa \cite{Igusa1} gave a set of generators of the ring $M(\Gamma_2)$.
The set consists of five generators
\[
\varphi_4,\quad \varphi_6,\quad \chi_{10},\quad \chi_{12},\quad \chi_{35},
\]
where $\varphi_k$ is the normalized Eisenstein series on $\Gamma_2$ and $\chi_k$ is a cusp form of weight $k$.
Moreover he showed that the even-weight generators $\varphi_4,\quad \varphi_6,\quad \chi_{10},\quad \chi_{12}$
are algebraically independent. Later, he extended the result to the integral case (\cite{Igusa2}). 
Namely, he gave a minimal set of generators over $\mathbb{Z}$ of the ring
\[
M(\Gamma_2)_{\mathbb{Z}}=\bigoplus_{k\in\mathbb{Z}}M_k(\Gamma_2)_{\mathbb{Z}}.
\]
The set of generators consists of fifteen modular forms including the following forms:
\[
X_4:=\varphi_4,\quad X_6:=\varphi_6,\quad X_{10}:=-2^{-2}\chi_{10},\quad X_{12}:=2^2\cdot 3\chi_{12},\quad
X_{35}:=2^2i\chi_{35}.
\]
Of course, these forms have rational integral Fourier coefficients under the following
normalization:
\begin{align*}
& a((0,0,0);X_4)=a((0,0,0);X_6)=1\\
& a((1,1,1);X_{10})=a((1,1,1);X_{12})=1\\
& a((2,3,-1);X_{35})=1.
\end{align*}
%%%%%%%%%%%%%%%%%%%%%%%%%%%%%%%%%%%%%%%%%%%%%%%%%%%%%%%%%%%%%%%%%%%%%%%%%%%%%%%%%%%%%%%%%%%%%
\subsection{Order and the $p$-minimum matrix}
\label{order}
We define a lexicographical order ``$\succ $'' for two different elements $T=(m,n,r)$ and $T'=(m',n',r')$ 
 of $\Lambda _2$ by
\begin{align*}
T\succ T'\Longleftrightarrow &(1)\ {\rm tr}(T)>{\rm tr}(T') \quad {\rm or}\quad  (2)\ {\rm tr}(T)={\rm tr}(T'),\ m>m'\quad {\rm or}\\ 
&(3)\ {\rm tr}(T)={\rm tr}(T'),\ m=m',\ r> r'. 
\end{align*}
Let $p$ be a prime and $\mathbb{Z}_{(p)}$ the local ring consisting of $p$-integral rational numbers.
For $F\in M_k(\Gamma_2)_{\mathbb{Z}_{(p)}}$, we define {\it the $p$-minimum matrix}  $m_p(F)$ of $F$ by 
\begin{align*}
m_p(F):=\min \{\;T\in L_2\;|\;a(T;F) \not \equiv 0 \pmod{p}\}, 
\end{align*}
where the ``min'' is defined in the sense of the above order. If $F\equiv 0 \pmod{p}$, then we define $m_p(F)=(\infty )$.
\begin{Rem}
The $p$-minimum matrices of Igusa's generators are
\[
m_p(X_4)=m_p(X_6)=(0,0,0),\quad m_p(X_{10})=m_p(X_{12})=(1,1,-1),\quad
m_p(X_{35})=(2,3,-1),
\]
for any prime number $p$.
\end{Rem}
The following properties are essential.
\begin{Lem}
\label{Lem1}
\noindent{(1)} $T_1\succ T_2$, $S_1\succ S_2$ implies $T_1+S_1\succ T_2+S_2$. \\
\noindent{(2)} $T_1\succ T_2$ implies $T_1\pm S\succ T_2 \pm S$. \\
\noindent{(3)} $T+S=T'+S'$, $T\succ T'$ implies $S\prec S'$. \\
\noindent{(4)} $m_p(F\cdot G)=m_p(F)+m_p(G)$.
\end{Lem}
\begin{proof}
(1), (2) Trivial. \\
(3) We use (2) without notice. By the assumption $T+S=T'+S'$, we have $T-T'=S'-S$. Then $0_2\prec T-T'=S'-S$ because of $T\succ T'$. 
Hence $S\prec S'$. \\
(4) Let $m_p(F)=T_0$ and $m_p(G)=T'_0$. Then for all $T\prec T_0$ (resp. $T\prec T'_0$), $a(T;F)\equiv 0 \pmod{p}$ and 
$a(T_0;F) \not \equiv 0 \pmod{p}$ (resp. $a(T;G) \equiv 0 \pmod{p}$ and $a(T_0;G))\not \equiv 0\pmod{p}$). Now, recall that the 
$T$-th Fourier coefficient $a(T;F\cdot G)$ of $F\cdot G$ is given by
\[
a(T;F\cdot G)=\sum _{\substack{S, S'\in L _2 \\ S+S'=T}}a(S;F)a(S';G). 
\]

If $T\prec T_0+T'_0$, then $T=S+S'\prec T_0+T'_0$ and hence $S\prec T_0$ or $S'\prec T'_0$ because of (1). In this case, 
$a(S;F)\equiv 0 \pmod{p}$ or 
$a(S';G) \equiv 0 \pmod{p}$. Therefore $a(S;F)a(S';G)\equiv 0 \pmod{p}$ for each $S$, $S'$ with $S+S'\prec T_0+S_0$. 
This implies $a(T;F) \equiv 0 \pmod{p}$ for all $T\prec T_0+S_0$. 

In order to complete the proof of (4), we need to prove that $a(T_0+T'_0;F\cdot G)\not \equiv 0 \pmod{p}$. If $S+S'=T_0+T'_0$, 
then we have by (3) that $S\prec T_0$, $S'\succ T'_0$ or $S\succ T_0$, $S'\prec T'_0$ or $S=T_0$, $S'=T'_0$. In the first two cases, 
since $a(S;F)\equiv 0 \pmod{p}$ or $a(S';G)\equiv 0 \pmod{p}$, we get $a(S;F)a(S';G)\equiv 0 \pmod{p}$. 
In the third case, $a(T_0;F)a(T'_0;G)\not \equiv 0 \pmod{p}$. Thus $a(T_0+T'_0;F\cdot G)\not \equiv 0 \pmod{p}$, 
namely $m_p(F\cdot G)=T_0+T'_0$. 
This completes the proof of (4).  

\end{proof}
%%%%%%%%%%%%%%%%%%%%%%%%%%%%%%%%%%%%%%%%%%%%%%%%%%%%%%%%%%%%%%%%%%%%%%%%%%%%%%%%%%%%%%%%%%%%%%%%%
\subsection{Sturm-type theorem}
\label{sturm}
A Sturm-type theorem for the Siegel modular forms of degree $2$ was presented in \cite{C-C-K}. We introduce the statement of this theorem for the case of level $1$.
%%%%%%%%%%%%%
\begin{Thm}[Choi-Choie-Kikuta \cite{C-C-K}] 
\label{C-C-K}
Let $p$ be a prime with $p\ge 5$ and $k$ an even positive integer. For $F\in M_k(\Gamma _2)_{\mathbb{Z}_{(p)}}$ with Fourier expansion 
\begin{align*}
F=\sum _{T\in L_2}a(T;F)q^T,
\end{align*}
we assume that $a((m,n,r);F) \equiv 0 \pmod{p}$ for all $m$, $n$, $r$ such that $0\leq m$, $n\leq \frac{k}{10}$ and $4mn-r^2\geq 0$. Then $F\equiv 0\pmod{p}$. 
\end{Thm}
We rewrite this theorem for later use:
%%%%%%%%%%%%%%%%%%%%%%%%%%
\begin{Thm}
\label{C-C-K2}
Let $p$ be a prime with $p\ge 5$. Assume that $F\in M_k(\Gamma _2)_{\mathbb{Z}_{(p)}}$ satisfies $m_p(F)\succ ([\frac{k}{10}],[\frac{k}{10}],r_0)$ 
for the maximum $r_0\in \mathbb{Z}$ such that $([\frac{k}{10}],[\frac{k}{10}],r_0)\in L_2$. Then $m_p(F)=(\infty )$, i.e., $F\equiv 0$ mod $p$. 
\end{Thm} 
\begin{proof}
The assertion follows immediately from the inclusion
\begin{align} 
\label{inc}
\Big\{ T \in L_2 \;\Big{|}\; T\preceq \Big( \Big[ \frac{k}{10} \Big],\Big[ \frac{k}{10} \Big],r_0\Big) \Big\} 
\supset \Big\{(m,n,r) \in L_2 \;\Big{|}\; m,n \le \frac{k}{10} \Big\}. 
\end{align}
\end{proof}
\begin{Rem}
In general, the converse of inclusion (\ref{inc}) is not true. For example $([\frac{k}{10}]+1,0,0)\prec ([\frac{k}{10}],[\frac{k}{10}],r_0)$ (for $k\geq 20$). 
We need a statement of this type to aid the proof of the next proposition.  
\end{Rem}
%%%%%%%%%%%%%%%%%%%%%%
In order to prove our main result, we need a Sturm-type theorem for {\it the odd-weight case}: 
\begin{Prop}
\label{Prop1}
Let $p$ be a prime with $p\ge 5$ and $k$ an odd positive integer. For $F\in M_k(\Gamma _2)_{\mathbb{Z}_{(p)}}$, 
we assume that $m_p(F)\succ ([\frac{k-35}{10}]+2,[\frac{k-35}{10}]+3,r_0-1)$, where $r_0\in \mathbb{Z}$ is the maximum number such 
that $([\frac{k-35}{10}],[\frac{k-35}{10}],r_0)\in L_2$. Then $m_p(F)=(\infty )$, namely $F\equiv 0 \pmod{p}$.   
\end{Prop}
\begin{Rem}
When $F\in M_k(\Gamma _2)_{\mathbb{Z}_{(p)}}$ is of odd weight, $X_{35}\cdot F\in M_{k+35}(\Gamma _2)_{\mathbb{Z}_{(p)}}$ is of even weight. Using 
Theorem \ref{C-C-K} directly, we have the following statement: If $a((m,n,r);F)\equiv 0 \pmod{p}$ for all $m$, $n$, $r$ 
such that $0\le m$, $n\le \frac{k+35}{10}$ and $4mn-r^2\ge 0$, then $F\equiv 0 \pmod{p}$. 

For our purpose, however, the estimation of Proposition \ref{Prop1} is better than this estimation.  
\end{Rem}

\begin{proof}[Proof of Proposition \ref{Prop1}]
First note that $M_k(\Gamma _2)_{\mathbb{Z}_{(p)}}=X _{35}M_{k-35}(\Gamma _2)_{\mathbb{Z}_{(p)}}$ for odd $k$. 
Hence there exists $G\in M_{k-35}(\Gamma _2)_{\mathbb{Z}_{(p)}}$ such that $F=X _{35}\cdot G$. Using (4) of Lemma \ref{Lem1}, 
we get $m_p (F)=m_p(X_{35})+m_p(G)$. Since $m_p(X_{35})=(2,3,-1)$, we have 
\begin{align*}
m_p(G)=m_p(F)-(2,3,-1)\succ \Big( \Big[ \frac{k-35}{10} \Big] ,\Big[\frac{k-35}{10} \Big] ,r_0\Big).
\end{align*} 
It should be noted that Lemma \ref{Lem1} (2) is used to get the last inequality.
 Since $G$ is of even weight, we can apply Theorem \ref{C-C-K2} to $G$. 
This shows that $F=X_{35}\cdot G \equiv 0 \pmod{p}$.    
\end{proof}
%%%%%%%%%%%%%%%%%%%%%%%%%%%%%%%%%%%%%%%%%%%%%%%%%%%%%%%%%%%%%%%%%%%%
\subsection{Theta operator}
\label{theta}
In \cite{B-N}, Serre used the theta operator $\theta$ on elliptic modular forms to develop the theory of $p$-adic modular forms:
\[
\theta=q\frac{d}{dq}:\; f=\sum a(t;f)q^t \longmapsto \theta (f):=\sum t\cdot a(t;f)q^t.
\]
Later the operator was generalized to the case of Siegel modular forms:
\[
\varTheta :\; F=\sum a(T;F)q^T \longmapsto \varTheta (F):=\sum \text{det}(T)\cdot a(T;F)q^T
\]
(e.g.cf. \cite{B-N})). Moreover the following fact was proven:
\begin{Thm}[B\"{o}cherer-Nagaoka \cite{B-N}] 
\label{B-N}
Assume that a prime $p$ satisfies $p\geq n+3$. Then for any Siegel modular form $F$ in $M_k(\Gamma_n)_{\mathbb{Z}_{(p)}}$,
there exists a Siegel cusp form $G$ in $S_{k+p+1}(\Gamma_n)_{\mathbb{Z}_{(p)}}$ satisfying
\[
\varTheta (F) \equiv G \pmod{p}.
\]
\end{Thm}
\noindent
{\it Example.} Under the notation in $\S$ \ref{generators}, we have
\[
\varTheta (X_6) \equiv 4X_{12} \pmod{5}.
\]
%%%%%%%%%%%%%%%%%%%%%%%%%%%%%%%%%%%%%%
\section{Main result}
\label{main}
On the basis of the previous preparation, we can now describe our main result.
\begin{Thm}
\label{mainTH}
Let $a(T;X_{35})$ denote the Fourier coefficient of $X_{35}$. 
If $\text{det}(T) \not\equiv 0 \pmod{23}$,
then
\[
a(T;X_{35}) \equiv 0 \pmod{23},
\]
or equivalently,
\[
\varTheta(X_{35}) \equiv 0 \pmod{23}.
\]
\end{Thm}
\begin{proof}
Our proof mainly depends on Proposition \ref{Prop1} and numerical calculation of the
Fourier coefficients of $X_{35}$. If we use the theta operator, this assertion is equivalent to showing
that
\[
\varTheta (X_{35}) \equiv 0 \pmod{23}.
\]
From Theorem \ref{B-N}, there exists a Siegel cusp form
$G\in S_{59}(\Gamma_2)_{\mathbb{Z}_{(23)}}$ such that
\[
\varTheta (X_{35}) \equiv G \pmod{23}.
\]
Therefore the proof is reduced to showing that
\begin{equation}
\label{last}
G \equiv 0 \pmod{23}.
\end{equation}
We now apply Proposition \ref{Prop1} to the form $G$. It then suffices to show that
\[
a((m,n,r);G) \equiv 0 \pmod{23}\quad \text{for}\;\; T=(m,n,r)\;\;\text{with}\;\;
\text{tr}(T)=m+n\leq 10.
\]
Since $a((m,n,r);G)=-a((n,m,r);G)$ for the odd-weight form $G$,
this statement is equivalent to 
\[
a((m,n,r);\varTheta (X_{35})) \equiv 0 \pmod{23}\quad \text{for} \;\;T=(m,n,r)\;\text{with}\;\; \text{tr}(T)=m+n\leq 9.
\]
We then write down the first part the Fourier expansion of $X_{35}$ following
the order introduced in $\S$ \ref{order}. For this, we set
\[
q_{jk}:=\text{exp}(2\pi i z_{jk})\quad \text{for}\quad
Z=\begin{pmatrix}z_{11}& z_{12}\\z_{12}&z_{22}\end{pmatrix}\in\mathbb{H}_2. 
\]
The terms corresponding to $T=(m,n,r)$ with $\text{tr}(T)=m+n\leq 9$ are as follows:
{\scriptsize
\begin{align*}
X_{35} &= (q_{12}^{-1}-q_{12})q_{11}^2q_{22}^3+(-q_{12}^{-1}+q_{12})q_{11}^3q_{22}^2\\
         &+(-q_{12}^{-3}-69q_{12}^{-1}+69q_{12}+q_{12}^3)q_{11}^2q_{22}^4
            +(q_{12}^{-3}+69q_{12}^{-1}-69q_{12}-q_{12}^3)q_{11}^4q_{22}^2\\
         &+(69q_{12}^{-3}+2277q_{12}^{-1}-2277q_{12}-69q_{12}^3)q_{11}^2q_{22}^5\\
         &+(q_{12}^{-5}-32384q_{12}^{-2}-129421q_{12}^{-1}+129421q_{12}+32384q_{12}^2
            -q_{12}^5)q_{11}^3q_{22}^4\\
         &+(-q_{12}^{-5}+32384q_{12}^{-2}+129421q_{12}^{-1}-129421q_{12}-32384q_{12}^2
            +q_{12}^5)q_{11}^4q_{22}^3\\
         &+(-69q_{12}^{-3}-2277q_{12}^{-1}+2277q_{12}+69q_{12}^3)q_{11}^5q_{22}^2\\
         &+(q_{12}^{-5}-2277q_{12}^{-3}-47702q_{12}^{-1}+47702q_{12}+2277q_{12}^3
            -q_{12}^5)q_{11}^2q_{22}^6\\
         &+(32384q_{12}^{-4}-2184448q_{12}^{-2}-3203072q_{12}^{-1}+3203072q_{12}
            +2184448q_{12}^2-32384q_{12}^4)q_{11}^3q_{22}^5\\
         &+(-32384q_{12}^{-4}+2184448q_{12}^{-2}+3203072q_{12}^{-1}-3203072q_{12}
            -2184448q_{12}^2+32384q_{12}^4)q_{11}^5q_{22}^3\\
         &+(-q_{12}^{-5}+2277q_{12}^{-3}+47702q_{12}^{-1}-47702q_{12}-2277q_{12}^3
            +q_{12}^5)q_{11}^6q_{22}^2\\
         &+(-69q_{12}^{-5}+47702q_{12}^{-3}+709665q_{12}^{-1}-709665q_{12}-47702q_{12}^3
            +69q_{12}^5)q_{11}^2q_{22}^7\\
         &+(-q_{12}^{-7}+129421q_{12}^{-5}+2184448q_{12}^{-4}+41321984q_{12}^{-2}
            +105235626q_{12}^{-1}\\
         &\quad -105235626q_{12}-41321984q_{12}^2-2184448q_{12}^4-129421q_{12}^5
          +q_{12}^7)q_{11}^3q_{22}^6\\                                                         
         &+(-69q_{12}^{-7}-32384q_{12}^{-6}+107121810q_{12}^{-3}-31380096q_{12}^{-2}
            +759797709q_{12}^{-1}\\
         &\quad -759797709q_{12}+31380096q_{12}^2-107121810q_{12}^3+32384q_{12}^6
            +69q_{12}^7)q_{11}^4q_{22}^5\\
         &+(69q_{12}^{-7}+32384q_{12}^{-6}-107121810q_{12}^{-3}+31380096q_{12}^{-2}
            -759797709q_{12}^{-1}\\
         &\quad +759797709q_{12}-31380096q_{12}^2+107121810q_{12}^3-32384q_{12}^6
            -69q_{12}^7)q_{11}^5q_{22}^4\\
         &+(q_{12}^{-7}-129421q_{12}^{-5}-2184448q_{12}^{-4}-41321984q_{12}^{-2}
            -105235626q_{12}^{-1}\\
         &\quad +105235626q_{12}+41321984q_{12}^2+2184448q_{12}^4+129421q_{12}^5
           -q_{12}^7)q_{11}^6q_{22}^3\\ 
         &+(69q_{12}^{-5}-47702q_{12}^{-3}-709665q_{12}^{-1}+709665q_{12}+47702q_{12}^3
            -69q_{12}^5)q_{11}^7q_{22}^2+\cdots\\         
\end{align*}
}
The Fourier coefficients different from $\pm 1$ are as follows:\\
{\scriptsize
$a((4,1,2);X_{35})=-69=-3\cdot\underline{23}$,\quad 
$a((5,1,2);X_{35})=2277=3^2\cdot 11\cdot\underline{23}$,\\
$a((4,1,3);X_{35})=-1294121=-17\cdot\underline{23}\cdot 331$,\quad 
$a((4,2,3);X_{35})=-32384=-2^7\cdot 11\cdot\underline{23}$,\\
$a((6,1,2);X_{35})=-47702=-2\cdot 17\cdot\underline{23}\cdot 61$,\quad 
$a((5,1,3);X_{35})=-3203072=-2^{13}\cdot 17\cdot\underline{23}$,\\
$a((5,2,3);X_{35})=-2184448=-2^8\cdot 7\cdot\underline{23}\cdot 53$,\quad 
$a((7,1,2);X_{35})=709665=3\cdot 5\cdot 11^2\cdot 17\cdot\underline{23}$,\\
$a((6,1,3);X_{35})=105235626=2\cdot 3\cdot\underline{23}\cdot 762577$,\quad 
$a((6,2,3);X_{35})=41321984=2^9\cdot 11^2\cdot\underline{23}\cdot 29$,\\
$a((5,1,4);X_{35})=759797709=3\cdot 11\cdot\underline{23}\cdot 29\cdot 34519$,\quad 
$a((5,2,4);X_{35})=-31380096=-2^7\cdot 3\cdot 11\cdot 17\cdot 19\cdot\underline{23}$,\\
$a((5,3,4);X_{35})=107121810=2\cdot 3\cdot 5\cdot 19\cdot\underline{23}\cdot 8171$.
}
\vspace{3mm}
\\
\quad All of these Fourier coefficients are divisible by 23. On the other hand,
if $a(T;X_{35})=\pm 1$ for $T$ in this range, then $\text{det}(T)=23/4 \equiv 0 \pmod{23}$.
This fact implies that
\[
a((m,n,r);\varTheta (X_{35})) \equiv 0 \pmod{23}
\]
for $T=(m,n,r)$ with $\text{tr}(T)=m+n\leq 9$. Therefore, we obtain
\[
a((m,n,r);G) \equiv 0 \pmod{23}
\]
for $T=(m,n,r)$ with $\text{tr}(T)=m+n\leq 9$. Consequently we have (\ref{last}).
This completes the proof of our theorem.
\end{proof}
%%%%%%%%%%%%%%%%%%%%%%%%
\begin{Rem} (1)\;
The converse statement of the theorem is not true in general. In fact
\[
a((1,6,1);X_{35})=0\quad\text{and}\quad \text{det}((1,6,1))=23/4 \equiv 0 \pmod{23}.
\]
Numerical examples of $a((m,n,r);X_{35})$ are found in \cite{A-I}, page 277.\\
(2)\; There are other ``modulo 23'' congruences for the Siegel modular forms
in \cite{B}, Satz 5,(a).
In that case, the congruence is concerned with the
Eisenstein lifting of the Ramanujan delta function.
\end{Rem} 

%%%%%%%%

Toshiyuki Kikuta\\ 
Department of Mathematics\\
Osaka Institute of Technology\\
5-16-1 Omiya, Asahi-ku, Osaka 535-8585, Japan\\
Email: kikuta84@gmail.com\\
\\
Hirotaka Kodama\\
Department of Mathematics\\
Kinki University\\
Higashi-Osaka, Osaka 577-8502, Japan\\
Email: kodama@math.kindai.ac.jp\\
\\
Shoyu Nagaoka\\
Department of Mathematics\\
Kinki University\\
Higashi-Osaka, Osaka 577-8502, Japan\\
Email: nagaoka@math.kindai.ac.jp

\end{document}